\documentclass[reqno]{amsart}
\usepackage{amssymb, amsmath, amsthm, amscd, mathrsfs}

\usepackage{paralist}
\usepackage{cite}
\usepackage{bigints}
\usepackage{dsfont}
\usepackage{empheq}
\usepackage{amssymb}
\usepackage{cases}
\usepackage{enumitem}
\allowdisplaybreaks

\usepackage{caption}
\usepackage{booktabs}

\usepackage{float}
\usepackage[ruled,vlined,linesnumbered,noresetcount]{algorithm2e}

\usepackage{hyperref}

\theoremstyle{plain}
\newtheorem{theorem}{Theorem}[section]
\theoremstyle{remark}
\newtheorem{remark}[theorem]{Remark}

\newtheorem{corollary}[theorem]{Corollary}
\newtheorem{lemma}[theorem]{Lemma}
\newtheorem{proposition}[theorem]{Proposition}
\newtheorem{definition}[theorem]{Definition}
\numberwithin{equation}{section}

\def \d {\mathrm{d}}

\title[Identification of source terms in heat equation] 
      {Identification of source terms in heat equation with dynamic boundary conditions}

\author{E. M. Ait Ben Hassi}
\address{E. M. Ait Ben hassi, Cadi Ayyad University, Faculty of Sciences Semlalia, LMDP, UMMISCO (IRD-UPMC), B.P. 2390, Marrakesh, Morocco}
\email{m.benhassi@gmail.com}	

\author{S. E. Chorfi}
\address{S. E. Chorfi, Cadi Ayyad University, Faculty of Sciences Semlalia, LMDP, UMMISCO (IRD-UPMC), B.P. 2390, Marrakesh, Morocco}
\email{chorphi@gmail.com}	

\author{L. Maniar}
\address{L. Maniar, Cadi Ayyad University, Faculty of Sciences Semlalia, LMDP, UMMISCO (IRD-UPMC), B.P. 2390, Marrakesh, Morocco}
\email{maniar@uca.ma}

\makeatletter 
\@namedef{subjclassname@2020}{%
  \textup{2020} Mathematics Subject Classification}
\makeatother

 

\begin{document}
\begin{abstract}
We study an inverse parabolic problem of identifying two source terms in heat equation with dynamic boundary conditions from a final time overdetermination data. Using a weak solution approach by Hasanov, the associated cost functional is analyzed, especially a gradient formula of the functional is proved and given in terms of the solution of an adjoint problem. Next, the existence and uniqueness of a quasi-solution are also investigated. Finally, the numerical reconstruction of some heat sources in a 1-D equation is presented to show the efficiency of the proposed algorithm.
\end{abstract}

\maketitle

\section{Introduction}
In this paper, we are interested in an inverse parabolic problem. It consists of identifying two source terms in a heat equation with dynamic boundary conditions, from a noisy measurement of the temperature at final time. 

Let $T>0$ be a fixed final time and let $\Omega \subset \mathbb{R}^N$ be a bounded domain ($N\geq 2$ is an integer) with boundary $\Gamma=\partial\Omega$ of class $C^2$. We denote $\Omega_T =(0,T)\times \Omega$ and $\Gamma_T =(0,T)\times \Gamma$. We consider the following heat equation
\begin{empheq}[left = \empheqlbrace]{alignat=2}
\begin{aligned}
&\partial_t y -d \Delta y+a(x)y = F(t,x), &&\qquad \text{in } \Omega_T , \\
&\partial_t y_{\Gamma} -\gamma \Delta_{\Gamma} y_{\Gamma}+d\partial_{\nu} y + b(x)y_{\Gamma} = G(t,x), &&\qquad \text{on } \Gamma_T, \\
&y_{\Gamma}(t,x) = y_{|\Gamma}(t,x), &&\qquad \text{on } \Gamma_T, \\
&(y,y_{\Gamma})\rvert_{t=0}=(y_0, y_{0,\Gamma}),   &&\qquad \Omega\times\Gamma \label{eq1to4}
\end{aligned}
\end{empheq}
for initial data $Y_0:=(y_0, y_{0,\Gamma})\in L^2(\Omega)\times L^2(\Gamma)$ and heat sources $F\in L^2(\Omega_T)$ and $G\in L^2(\Gamma_T)$. The diffusion coefficients are positive constants $d, \gamma > 0$, and the spatial potentials are such that $a\in L^\infty(\Omega)$ and $b\in L^\infty(\Gamma)$. $\Delta=\Delta_x$ denotes the standard Laplace operator with respect to the space variable. By $y_{\mid \Gamma},$ one denotes the trace of $y$, while the normal derivative is denoted by $\partial_{\nu} y:=(\nabla y \cdot \nu)_{|\Gamma}$, where $\nu(x)$ stands for the unit outward normal vector to $\Gamma$ at $x$. The tangential gradient $\nabla_\Gamma y$ is given by $\nabla_{\Gamma} y=\nabla y-\left(\partial_{\nu} y\right) \nu$. Let $\mathrm{g}$ be the natural Riemannian metric on $\Gamma$ inherited from $\mathbb{R}^N$. The Laplace-Beltrami operator $\Delta_\Gamma$ (with respect to $\mathrm{g}$) is given locally by
\begin{equation*}
\Delta_\Gamma=\frac{1}{\sqrt{|\mathrm{g}|}} \sum_{i,j=1}^{N-1} \frac{\partial}{\partial x^i} \left(\sqrt{|\mathrm{g}|}\, \mathrm{g}^{ij} \frac{\partial}{\partial x^j}\right),
\end{equation*}
where $\mathrm{g}=\left(\mathrm{g}_{ij}\right)$ is the metric tensor corresponding to $\mathrm{g}$, $\mathrm{g}^{-1}=\left(\mathrm{g}^{ij}\right)$ its inverse and $|\mathrm{g}|=\det\left(\mathrm{g}_{ij}\right)$. In the sequel, we mainly use the following surface divergence formula
\begin{equation}\label{sdt}
\int_\Gamma \Delta_\Gamma y\, z \,\d S =- \int_\Gamma \langle \nabla_\Gamma y, \nabla_\Gamma z\rangle_\Gamma \,\d S, \quad y\in H^2(\Gamma), \, z\in H^1(\Gamma),
\end{equation}
where $\langle \cdot, \cdot \rangle_\Gamma$ is the Riemannian inner product of tangential vectors on $\Gamma$.

Parabolic equations with dynamic boundary conditions have received a lot of attention of many researchers in the last years \cite{BCMO'20, FGGR'02, KM'19, KMMR'19, MMS'17, MZ'05}, since they appear in several fields of applications including chemical engineering such as chemical reactors, the chemistry of colloids and special flows in hydrodynamics \cite{VV'83}. The scope of applications also includes heat transfer problems where the diffusion takes place between a solid and a fluid in motion \cite{La'32}. We refer to the seminal paper \cite{Go'06}, where the physical derivation and interpretation of dynamic boundary conditions are discussed.

A variety of analytical and numerical techniques have been proposed for solving inverse problems in partial differential equations (PDEs) with static boundary conditions (Dirichlet, Neumann or Robin conditions); see, for instance, \cite{BK'81, Er'15, Ha'07, IS'90}. A well-known theoretical technique is the Carleman estimate, which is a priori weighted inequality estimating the solution of a PDE and its derivatives via the associated differential operator. Such an estimate was introduced, for the first time, in the study of multidimensional inverse problems by Bukhgeim and Klibanov \cite{BK'81}, and became a powerful tool to establish uniqueness and stability results. Recently in \cite{ACMO'20}, the authors have applied this method in the study of an inverse source problem for general parabolic equations with dynamic boundary conditions from interior measurements. They established a Lipschitz stability estimate for the source terms. In \cite{ACM'21}, we have considered an inverse problem of radiative potentials and initial temperatures for the same equation. We have proven a Lipschitz stability estimate for the potentials and then obtained a logarithmic stability result for the initial conditions by a logarithmic convexity method. It is worthwhile to mention that Carleman estimates are also important in the numerical study of inverse problems \cite{KL'13}.

In the context of numerical studies, Hasanov \cite{Ha'07} developed a weak solution approach to study the simultaneous determination of source terms in heat equation with static boundary conditions from the final overdetermination, Dirichlet or Neumann types output measured data \cite{Ha'11, HOA'11, KE'15}. The underlying method relies on reformulating a quasi-solution problem as a minimization problem of Tikhonov functional, combined with an adjoint problem approach introduced by DuChateau \cite{Du'96, DTB'04}. This approach provides a monotone iteration scheme to reconstruct unknown parameters in ill-posed problems, which implies fast numerical results. It has also been successfully applied to source terms in the cantilevered beam equation \cite{Ha'09}, Lotka Volterra system \cite{GBB'12}, and recently for thermal conductivity and radiative potential in heat equation from Dirichlet and Neumann boundary measured outputs \cite{Ha'20}. We refer the interested readers to the recent book by Hasanov and Romanov \cite{HR'17}, which presents a systematic study of mathematical and numerical methods used in inverse problems within this framework.

In the present work, we extend the previous technique for the determination of two source terms acting in both, in the domain and on the boundary, in heat equation with dynamic boundary conditions. In this context, we mention \cite{Sl'15} for the identification of a time-dependent source term from the knowledge of a space average using the backward Euler method. Also, in \cite{Ma'18}, the author deals with a time-dependent source from an integral overdetermination condition in a 1-D heat equation by using the generalized Fourier method. These two papers dealt with a basic dynamic boundary condition that only contains time and normal derivatives on the boundary (known as Wentzell/Ventcel boundary condition). To the best of our knowledge, the study of such problem by a weak solution approach has not been addressed in the literature, and only equations with classical boundary conditions (Dirichlet, Neumann and Robin) are considered. Here, we deal with a multidimensional heat system of two equations coupled through the boundary via the normal derivative, and which contains a surface diffusion on the boundary. Thus, the present problem requires a careful analysis and suitable combination between the two equations in order to obtain the desired results. We also study numerically the reconstruction for a heat source in the 1-D equation from a noisy terminal data.

The rest of the paper is organized as follows: in Section \ref{sec2}, we recall the well-posedness and regularity results of the system \eqref{eq1to4}. In Section \ref{sec3}, we study the minimization problem associated to quasi-solutions. We also infer an explicit gradient formula for the Tikhonov functional via the solution of an appropriate adjoint problem. Then, we establish the Lipschitz continuity of the gradient of the cost functional. In Section \ref{sec4}, we discuss the existence and uniqueness of a quasi-solution. Finally, in Section \ref{sec5}, the gradient formula is implemented via Landweber scheme for the numerical reconstruction of an unknown source term in 1-D equation. Section \ref{sec6} is devoted to some conclusions.

\section{Wellposedness and regularity of the solution}\label{sec2}
We recall some results on the wellposedness and regularity of the solution to system \eqref{eq1to4} needed in the sequel. The reader can refer to \cite{MMS'17} for detailed proofs.

We first introduce the following real spaces
$$\mathbb{L}^2:=L^2(\Omega, \d x)\times L^2(\Gamma, \d S)\qquad \text{ and} \qquad \mathbb{L}^2_T:=L^2(\Omega_T)\times L^2(\Gamma_T).$$
$\mathbb{L}^2$ and $\mathbb{L}^2_T$ are Hilbert spaces equipped with the scalar products given by
\begin{align*}
\langle (y,y_\Gamma),(z,z_\Gamma)\rangle_{\mathbb{L}^2} &=\langle y,z\rangle_{L^2(\Omega)} +\langle y_\Gamma,z_\Gamma\rangle_{L^2(\Gamma)},\\
\langle (y,y_\Gamma),(z,z_\Gamma)\rangle_{\mathbb{L}^2_T} &=\langle y,z\rangle_{L^2(\Omega_T)} +\langle y_\Gamma,z_\Gamma\rangle_{L^2(\Gamma_T)},
\end{align*}
respectively, where we denoted the Lebesgue measure on $\Omega$ by $\d x$ and the surface measure on $\Gamma$ by $\d S$. We also define the space $$\mathbb{H}^k:=\left\{(y,y_\Gamma)\in H^k(\Omega)\times H^k(\Gamma)\colon y_{|\Gamma} =y_\Gamma \right\} \text{ for } k=1,2,$$ equipped with the standard product norm.

For the regularity of the solution, we consider the following spaces
$$\mathbb{E}_1(t_0,t_1):=H^1\left(t_0,t_1 ;\mathbb{L}^2\right) \cap L^2\left(t_0,t_1 ;\mathbb{H}^2\right) \text{  for } t_1 >t_0 \text{ in } \mathbb{R},$$
$$\mathbb{E}_2(t_0,t_1):=H^1\left(t_0,t_1;\mathbb{H}^2\right) \cap H^2\left(t_0,t_1;\mathbb{L}^2\right) \text{  for } t_1 >t_0 \text{ in } \mathbb{R}.$$
In particular,
$$\mathbb{E}_1 := \mathbb{E}_1(0,T) \;\text{ and }\; \mathbb{E}_2:= \mathbb{E}_2(0,T).$$
We rewrite \eqref{eq1to4} in the following abstract form
\begin{equation}\label{acp}
\text{(ACP)	} \; \begin{cases}
\hspace{-0.1cm} \partial_t Y=\mathcal{A} Y+ \mathcal{F}, \quad 0<t \le T, \nonumber\\
\hspace{-0.1cm} Y(0)=Y_0:=(y_0, y_{0,\Gamma}), \nonumber
\end{cases}
\end{equation}
where $Y:=(y,y_{\Gamma})$, $\mathcal{F}=(F,G)$ and the linear operator $\mathcal{A} \colon D(\mathcal{A}) \subset \mathbb{L}^2 \longrightarrow \mathbb{L}^2$ is given by
\begin{equation*}
\mathcal{A}=\begin{pmatrix} d\Delta - a & 0\\ -d\partial_\nu & \gamma \Delta_\Gamma - b\end{pmatrix}, \qquad \qquad D(\mathcal{A})=\mathbb{H}^2.
\end{equation*}
The operator $\mathcal{A}$ generates an analytic $C_0$-semigroup $\left(\mathrm{e}^{t\mathcal{A}}\right)_{t\geq 0}$ on $\mathbb{L}^2$, see \cite{MMS'17} for more details.
\bigskip

In the sequel, we will adopt the following notions for the solutions to system \eqref{eq1to4}.
\begin{definition}
Let $Y_0\in \mathbb{L}^2$ and $(F,G) \in \mathbb{L}^2_T$.
\begin{enumerate}[label=(\alph*),leftmargin=*]
\item A strong solution of \eqref{eq1to4} is a function $Y:=(y, y_{\Gamma}) \in \mathbb{E}_1$ fulfilling \eqref{eq1to4} in $L^2(0,T; \mathbb{L}^2)$.
\item A mild solution of \eqref{eq1to4} is a function $Y:=(y, y_{\Gamma}) \in C([0,T]; \mathbb{L}^2)$ satisfying, for $t\in [0,T]$,
$$Y(t,\cdot)=\mathrm{e}^{t\mathcal{A}} Y_0 + \int_0^t \mathrm{e}^{(t-\tau)\mathcal{A}} [F(\tau, \cdot), G(\tau, \cdot)] \,\d\tau.$$
\item A distributional (weak) solution of \eqref{eq1to4} is a function 
$Y:= (y, y_{\Gamma})\in L^2(0,T;\mathbb{L}^2)$ such that for all $(\varphi,\varphi_\Gamma) \in \mathbb{E}_1$ with $\varphi(T,\cdot) =\varphi_\Gamma(T,\cdot)= 0$ we have
\begin{align*}
 & \int_{\Omega_T} y(-\partial_t \varphi -  d\Delta \varphi +a\varphi)\,\d x \,\d t  
  + \int_{\Gamma_T} y_\Gamma(-\partial_t\varphi_\Gamma- \gamma \Delta_\Gamma \varphi_\Gamma + d\partial_\nu \varphi 
        + b\varphi_\Gamma)\,\d S \,\d t \\
 & = \int_{\Omega_T} F\varphi \,\d x \,\d t + \int_{\Gamma_T} G\varphi_\Gamma \,\d S \,\d t 
   + \int_\Omega y_0 \varphi(0,\cdot)\,\d x + \int_\Gamma y_{0,\Gamma} \varphi_\Gamma(0,\cdot) \,\d S.
\end{align*}
\end{enumerate}
\end{definition} 

The following proposition shows the $\mathbb{L}^2$-regularity for the system \eqref{eq1to4} and highlights the connections between different types of solutions. For the proof, see \cite[Propositions 2.4 and 2.5]{MMS'17}.
\begin{proposition}\label{prop2}
Let $F\in L^2(\Omega_T)$ and $G\in L^2(\Gamma_T)$.
\begin{enumerate}[label=(\roman*),leftmargin=*]
\item For all $Y_0:=(y_0,y_{0,\Gamma})\in \mathbb{H}^1$, there exists a unique strong solution of \eqref{eq1to4} such that $Y:=(y,y_{\Gamma})\in \mathbb{E}_1$. Moreover, there exists a positive constant $C=C(\|a\|_\infty, \|b\|_\infty)$ such that
\begin{equation}\label{eqprop2}
\|Y\|_{\mathbb{E}_{1}} \leq C\left(\|F\|_{L^2(\Omega_T)} + \|G\|_{L^2(\Gamma_T)} + \|Y_0\|_{\mathbb{H}^1}\right).
\end{equation}
\item For all $Y_0:=(y_0,y_{0,\Gamma})\in \mathbb{L}^2$, there exists a unique mild solution of \eqref{eq1to4} $Y:=(y,y_{\Gamma})\in C([0,T]; \mathbb{L}^2)$ such that for all $\tau \in (0,T)$, $$Y\in \mathbb{E}_1(\tau, T):=H^1\left(\tau,T;\mathbb{L}^2\right) \cap L^2\left(\tau,T;\mathbb{H}^2\right).$$ 
Moreover, if $\mathcal{F}=(F,G)\in H^1(0,T; \mathbb{L}^2)$, then for all $\tau \in (0,T)$, we have
$$Y\in \mathbb{E}_2(\tau, T):=H^1\left(\tau,T;\mathbb{H}^2\right) \cap H^2\left(\tau,T;\mathbb{L}^2\right),$$
with initial data $Y(\tau)$.
\item A function $Y$ is a distributional solution of \eqref{eq1to4} if and only if it is a mild solution.
\end{enumerate}
\end{proposition}

\section{Fréchet differentiability and gradient formula of the cost functional}\label{sec3}
In this section, we consider the following inverse source problem.

\noindent\textbf{Inverse Source Problem (ISP).} A couple of source terms $(F,G)\in \mathbb{L}^2_T$ in \eqref{eq1to4} is unknown and needs to be recovered from the final temperature at $T$, namely,
$$Y_{T}:=\left(y(T,\cdot),y_\Gamma(T,\cdot)\right)\in \mathbb{L}^2,$$
which is not necessarily smooth due to the numerical noise.

Let $Y(t,\cdot,\mathcal{F})$ be the mild solution of \eqref{eq1to4} corresponding to the source terms $\mathcal{F}=(F,G)\in \mathbb{L}^2_T$. We introduce the input-output operator $\Psi \colon \mathbb{L}^2_T \longrightarrow \mathbb{L}^2$ defined as follows
$$(\Psi \mathcal{F})(\cdot)=Y_T(\cdot):=Y(T,\cdot,\mathcal{F}) \qquad\text{ on } \;\Omega\times\Gamma.$$
Then the inverse source problem (ISP) can be reformulated as the following operator equation
\begin{equation}\label{2eq2.3}
\Psi \mathcal{F}=Y_T, \quad Y_T \in \mathbb{L}^2.
\end{equation}

The following lemma is needed in the sequel.
\begin{lemma}\label{lemCont}
Let $Y_0 \in \mathbb{L}^2$ and let $Y$ be the mild solution of \eqref{eq1to4} corresponding to $\mathcal{F}=(F,G)$. Then the solution map $\mathcal{F}\longmapsto Y$ is continuous from $H^1\left(0,T; \mathbb{L}^2\right)$ to $C\left([0,T];\mathbb{L}^2\right)\cap L^2\left(0,T; \mathbb{H}^1\right)$. 
\end{lemma}

\begin{proof}
Let $Y_0 \in \mathbb{H}^2$ and $\delta \mathcal{F}$ be a small variation of $\mathcal{F}$ such that $\mathcal{F}+\delta \mathcal{F} \in \mathcal{U}$. Consider $\delta Y :=Y^\delta -Y$, where $Y^\delta$ is the mild solution of \eqref{eq1to4} corresponding to $\mathcal{F}^\delta :=\mathcal{F}+\delta \mathcal{F}$. Then $\delta Y \in C^1\left([0,T],\mathbb{L}^2\right)\cap C\left(0,T; \mathbb{H}^2\right)$ and satisfies the following system
\begin{empheq}[left = \empheqlbrace]{alignat=2}
\begin{aligned}
&\partial_t (\delta y) -d \Delta (\delta y)+ a(x)(\delta y) = \delta F(t,x), &\qquad \text{in } \Omega_T , \\
&\partial_t (\delta y_{\Gamma}) -\gamma \Delta_{\Gamma} (\delta y_{\Gamma})+d\partial_{\nu} y + b(x)(\delta y_{\Gamma}) = \delta G(t,x), &\qquad \text{on } \Gamma_T, \\
&\delta y_{\Gamma}(t,x) = (\delta y)_{|\Gamma}(t,x), &\qquad \text{on } \Gamma_T, \\
&(\delta y, \delta y_{\Gamma})\rvert_{t=0}=(0,0),   &\qquad \Omega \times\Gamma,
\label{0deq1to4}
\end{aligned}
\end{empheq}
where
\begin{align*}
\delta y(T, \cdot, \mathcal{F})&= y(T, \cdot, \mathcal{F}+ \delta \mathcal{F})- y(T, \cdot, \mathcal{F}),\\
\delta y_\Gamma(T, \cdot, \mathcal{F}) &= y_\Gamma(T, \cdot, \mathcal{F}+ \delta \mathcal{F})-y_\Gamma(T, \cdot, \mathcal{F}).
\end{align*}
Multiplying $\eqref{0deq1to4}_1$ by $\delta y$, $\eqref{0deq1to4}_2$ by $\delta y_\Gamma$ and using Green and surface divergence formulae, we obtain
\begin{align*}
& \frac{1}{2} \partial_t \left( \int_\Omega |\delta y|^2 \,\d x\right) +d\int_\Omega |\nabla(\delta y)|^2 \,\d x -\frac{1}{2} \int_\Gamma d\partial_\nu (|\delta y|^2) \,\d S + \int_\Omega a |\delta y|^2 \,\d x \nonumber\\
& \qquad =\int_\Omega \delta F \delta y \,\d x,\\
& \frac{1}{2} \partial_t \left( \int_\Gamma |\delta y_\Gamma|^2 \,\d S\right) +\gamma  \int_\Gamma |\nabla_\Gamma(\delta y_\Gamma)|^2 \,\d S +\frac{1}{2} \int_\Gamma d\partial_\nu (|\delta y|^2) \,\d S + \int_\Gamma b |\delta y_\Gamma|^2 \,\d S \nonumber\\
& \qquad =\int_\Gamma \delta G \delta y_\Gamma \,\d S.
\end{align*}
Adding up the last two equalities, we arrive at
\begin{align*}
&\frac{1}{2} \partial_t \left( \int_\Omega |\delta y|^2 \,\d x + \int_\Gamma |\delta y_\Gamma|^2 \,\d S\right) + d\int_\Omega |\nabla(\delta y)|^2 \,\d x + \gamma  \int_\Gamma |\nabla_\Gamma(\delta y_\Gamma)|^2 \,\d S \nonumber\\
&\qquad + \int_\Omega a |\delta y|^2 \,\d x + \int_\Gamma b |\delta y_\Gamma|^2 \,\d S \nonumber\\
&=  \int_\Omega \delta F \delta y \,\d x + \int_\Gamma \delta G \delta y_\Gamma \,\d S.
\end{align*}
Then, using the Cauchy-Schwarz inequality, we have
\begin{align}
&\frac{1}{2} \frac{\d}{\d t} \|\delta Y(t)\|^2_{\mathbb{L}^2} +\min(d,\gamma)\|(\nabla y(t), \nabla_\Gamma y_\Gamma(t))\|^2_{\mathbb{L}^2}\nonumber\\
& \qquad\leq \langle \delta Y, \delta\mathcal{F}\rangle_{\mathbb{L}^2} +\max(\|a\|_\infty, \|b\|_\infty)\|\delta Y(t)\|^2_{\mathbb{L}^2}\nonumber\\
& \qquad\leq \|\delta Y(t)\|_{\mathbb{L}^2} \|\delta \mathcal{F}(t)\|_{\mathbb{L}^2}+\max(\|a\|_\infty, \|b\|_\infty)\|\delta Y(t)\|^2_{\mathbb{L}^2}\nonumber\\
& \qquad\leq \frac{1}{2} \|\delta Y(t)\|^2_{\mathbb{L}^2} + \frac{1}{2} \|\delta \mathcal{F}(t)\|^2_{\mathbb{L}^2}+\max(\|a\|_\infty, \|b\|_\infty)\|\delta Y(t)\|^2_{\mathbb{L}^2}, \label{02eq2.39}
\end{align}
where $\mathcal{F}=(F,G)$. Consequently,
\begin{equation*}
\frac{\d}{\d t} \|\delta Y(t)\|^2_{\mathbb{L}^2} \leq (1+2\max(\|a\|_\infty, \|b\|_\infty))\|\delta Y(t)\|^2_{\mathbb{L}^2} + \|\delta \mathcal{F}(t)\|^2_{\mathbb{L}^2} .
\end{equation*}
By Gronwall inequality, we deduce
\begin{align}
\|\delta Y(t)\|^2_{\mathbb{L}^2} &\leq \mathrm{e}^{(1+2\max(\|a\|_\infty, \|b\|_\infty))T} \left(\|\delta Y(0)\|^2_{\mathbb{L}^2}+ \|\delta \mathcal{F}\|^2_{\mathbb{L}^2_T}\right)\nonumber\\
&= \mathrm{e}^{(1+2\max(\|a\|_\infty, \|b\|_\infty))T} \|\delta \mathcal{F}\|^2_{\mathbb{L}^2_T} \label{02eq2.40}
\end{align} 
for every $0\leq t\leq T$. On the other hand, from \eqref{02eq2.39} and \eqref{02eq2.40}, we have
\begin{align*}
&\frac{\d}{\d t} \|\delta Y(t)\|^2_{\mathbb{L}^2}+2\min(d,\gamma)\|(\nabla y(t), \nabla_\Gamma y_\Gamma(t))\|^2_{\mathbb{L}^2}\nonumber\\
& \qquad\leq \left(1+2\max(\|a\|_\infty, \|b\|_\infty)\right)\|\delta Y(t)\|^2_{\mathbb{L}^2} + \|\delta \mathcal{F}(t)\|^2_{\mathbb{L}^2}.\nonumber\\
& \qquad  \leq \left(1+2\max(\|a\|_\infty, \|b\|_\infty)\right) \mathrm{e}^{(1+2\max(\|a\|_\infty, \|b\|_\infty))T} \|\delta \mathcal{F}\|^2_{\mathbb{L}^2_T}+ \|\delta\mathcal{F}(t)\|^2_{\mathbb{L}^2}.
\end{align*}
Integrating between $0$ and $T$ and using the fact that $\delta Y(0)=0$, we deduce
\begin{align*}
&\|\delta Y(T)\|^2_{\mathbb{L}^2}+ 2\min(d,\gamma) \int_0^T \|(\nabla y(t), \nabla_\Gamma y_\Gamma(t))\|^2_{\mathbb{L}^2} \,\d t \nonumber\\
&\qquad\leq \left(1+2\max(\|a\|_\infty, \|b\|_\infty)\right)T \mathrm{e}^{(1+2\max(\|a\|_\infty, \|b\|_\infty))T}\|\delta\mathcal{F}\|^2_{\mathbb{L}^2_T} + \|\delta\mathcal{F}\|^2_{\mathbb{L}^2_T}.
\end{align*}
Then we arrive at
\begin{align*}
\int_0^T \|(\nabla y(t), \nabla_\Gamma y_\Gamma(t))\|^2_{\mathbb{L}^2} \,\d t \leq C_T \|\delta \mathcal{F}\|^2_{\mathbb{L}^2_T},
\end{align*}
with
$$C_T :=\frac{(1+2\max(\|a\|_\infty, \|b\|_\infty))T \mathrm{e}^{(1+2\max(\|a\|_\infty, \|b\|_\infty))T} +1}{2 \min(d,\gamma)}.$$
The latter inequality and \eqref{02eq2.40} imply that
\begin{align*}
\sup_{t\in [0,T]} \|\delta Y(t)\|^2_{\mathbb{L}^2} + \int_0^T \|\delta Y(t)\|^2_{\mathbb{H}^1} \,\d t \leq C \|\delta \mathcal{F}\|^2_{\mathbb{L}^2_T}
\end{align*}
for some generic constant $C=C(T,\|a\|_\infty,\|b\|_\infty, d,\gamma)>0$. Hence,
\begin{align*}
\|\delta Y\|^2_{C([0,T];\mathbb{L}^2)} + \|\delta Y\|^2_{L^2(0,T;\mathbb{H}^1)} \leq C \|\delta \mathcal{F}\|^2_{H^1(0,T; \mathbb{L}^2)}. 
\end{align*}
Since $\mathbb{H}^2$ is dense in $\mathbb{L}^2$, the same inequality holds for any $Y_0 \in \mathbb{L}^2$. This ends the proof.
\end{proof}

The following lemma highlights the compactness of the input-output operator $\Psi \colon \mathbb{L}^2_T \longrightarrow \mathbb{L}^2$. By linearity, we may assume here that $Y_0:=(y_0,y_{0,\Gamma})=(0,0)$.
\begin{lemma}\label{2lem2.4}
The input-output operator 
\begin{equation*}
\Psi \colon \mathbb{L}^2_T \ni \mathcal{F} \longmapsto Y(T,\cdot,\mathcal{F}) \in \mathbb{L}^2
\end{equation*}
is compact.
\end{lemma}

\begin{proof}
Consider $(\mathcal{F}_n)_{n\in \mathbb{N}}$ a bounded sequence in $\mathbb{L}^2_T$. By estimate \eqref{eqprop2}, the sequence of associated strong solutions $(Y(\cdot,\cdot,\mathcal{F}_n))_{n\in \mathbb{N}}$ of \eqref{eq1to4} is bounded in $\mathbb{E}_1$. Since the trace space of $\mathbb{E}_1$ at $t=T$ equals $\mathbb{H}^1$ (see \cite[Proposition 2.2]{MMS'17}), the sequence $(Y_n)_{n\in \mathbb{N}}$ defined by $Y_n:=Y(T, \cdot, \mathcal{F}_n)$ is bounded in $\mathbb{H}^1$. By the compact embedding $\mathbb{H}^1 \hookrightarrow \mathbb{L}^2$ (see \cite{MMS'17}), there exists a subsequence of $(Y_n)_{n\in \mathbb{N}}$ which converges in $\mathbb{L}^2$. Thus, the input-output operator $\Psi$ is compact.
\end{proof}

It follows from Lemma \ref{2lem2.4} that the inverse problem \eqref{2eq2.3} is ill-posed (in the sense of Hadamard). A quasi-solution $\mathcal{F}_*$ of the ill-posed problem \eqref{2eq2.3} is defined as a solution of the following minimization problem
\begin{align}
&\mathcal{J}(\mathcal{F}_*)= \inf_{\mathcal{F}\in \mathcal{U}} \mathcal{J}(\mathcal{F}), \label{neq2.5}\\
\mathcal{J}(\mathcal{F})&=\frac{1}{2} \|Y(T, \cdot,\mathcal{F})-Y_{T}^\delta\|_{\mathbb{L}^2}^2, \qquad \mathcal{F}\in \mathcal{U}, \label{neq2.6}
\end{align}
where $Y_{T}^\delta=(y_{T}^\delta, y_{T, \Gamma}^\delta)$ is a noisy data of $Y_{T}$ such that $\|Y_{T}-Y_{T}^\delta\|\leq \delta$ for $\delta \geq 0$, and 
\begin{equation*}
\mathcal{U}:=\{\mathcal{F}=(F,G)\in H^1(0,T; \mathbb{L}^2) \colon \|\mathcal{F}\|_{H^1(0,T; \mathbb{L}^2)} \leq R\}
\end{equation*}
is the set of admissible source terms. Clearly, $\mathcal{U}$ is a bounded, closed and convex subset of $H^1(0,T; \mathbb{L}^2)$, for any fixed $R>0$.

Due to the ill-posedness of \eqref{2eq2.3}, we usually use a Tikhonov regularization approach and consider instead the following regularized functional
\begin{align*}
\mathcal{J}_\varepsilon(\mathcal{F})&=\frac{1}{2} \|Y(T, \cdot,\mathcal{F})-Y_{T}^\delta\|_{\mathbb{L}^2}^2 +\frac{\varepsilon}{2} \|\mathcal{F}\|_{\mathbb{L}^2_T}^2, \qquad \mathcal{F}\in \mathcal{U}, 
\end{align*}
where $\varepsilon >0$ is the regularizing parameter.

Next, we derive a gradient formula for $\mathcal{J}$ via the mild solution $\Phi=(\varphi, \varphi_\Gamma)$ of an adjoint system.
\begin{proposition}\label{2prop2.6}
The cost functional $\mathcal{J}$ is Fr\'echet differentiable and its gradient at each $\mathcal{F}\in \mathcal{U}$ is given by
\begin{equation}\label{2eq2.11}
\mathcal{J}'(\mathcal{F})= \Phi ,
\end{equation}
where $\Phi(t,\cdot,\mathcal{F})=(\varphi, \varphi_\Gamma)$ is the mild solution of the following adjoint system
\begin{empheq}[left = \empheqlbrace]{alignat=2}
\begin{aligned}
&-\partial_t \varphi -d \Delta \varphi +a(x)\varphi = 0, &&\qquad \text{in } \Omega_T , \\
&-\partial_t \varphi_{\Gamma} -\gamma \Delta_{\Gamma} \varphi_{\Gamma} +d\partial_{\nu} \varphi +b(x)\varphi_{\Gamma} = 0, &&\qquad\text{on } \Gamma_T, \\
&\varphi_{\Gamma}(t,x) = \varphi_{|\Gamma}(t,x), && \qquad\text{on } \Gamma_T, \\
&\varphi\rvert_{t=T} =y(T, \cdot, \mathcal{F})-y_{T}^\delta,   &&\qquad \text{in } \Omega ,\\
&\varphi_{\Gamma}\rvert_{t=T} = y_\Gamma(T, \cdot, \mathcal{F})-y_{T, \Gamma}^\delta, && \qquad\text{on } \Gamma .
\label{aeq1to5}
\end{aligned}
\end{empheq}
\end{proposition}
\begin{proof}
We assume that $\mathcal{F}, \mathcal{F}+ \delta \mathcal{F}\in \mathcal{U}$. Let us calculate the difference
$$\delta \mathcal{J}(\mathcal{F}):=\mathcal{J}(\mathcal{F}+\delta \mathcal{F}) -\mathcal{J}(\mathcal{F}).$$
We have
\begin{align*}
\delta \mathcal{J}(\mathcal{F})&= \frac{1}{2} \|Y(T, \cdot, \mathcal{F}+ \delta \mathcal{F})-Y_{T}^\delta\|_{\mathbb{L}^2}^2 -\frac{1}{2} \|Y(T, \cdot,\mathcal{F})-Y_{T}^\delta\|_{\mathbb{L}^2}^2, \nonumber\\
&= \frac{1}{2} \left(\|y(T, \cdot, \mathcal{F}+ \delta \mathcal{F})-y_{T}^\delta\|_{L^2(\Omega)}^2 + \|y_\Gamma(T, \cdot, \mathcal{F}+ \delta \mathcal{F})-y_{T, \Gamma}^\delta\|_{L^2(\Gamma)}^2\right) \nonumber\\
& \qquad - \frac{1}{2} \left(\|y(T, \cdot, \mathcal{F})-y_{T}^\delta\|_{L^2(\Omega)}^2 + \|y_\Gamma(T, \cdot, \mathcal{F})-y_{T, \Gamma}^\delta\|_{L^2(\Gamma)}^2\right) \nonumber\\
&= \frac{1}{2} \int_\Omega \left[(y(T, x, \mathcal{F}+ \delta \mathcal{F})-y_{T}^\delta(x))^2  - (y(T, x, \mathcal{F})-y_{T}^\delta(x))^2 \right] \,\d x \\
& \quad + \frac{1}{2} \int_\Gamma \left[(y_\Gamma(T, x, \mathcal{F}+ \delta \mathcal{F})-y_{T, \Gamma}^\delta(x))^2 - (y_\Gamma(T, x, \mathcal{F})-y_{T, \Gamma}^\delta)^2 \right] \,\d S.
\end{align*}
Using the identity
$$\frac{1}{2} \left[(x-z)^2-(y-z)^2\right]=(y-z)(x-y)+\frac{1}{2}(x-y)^2, \; x,y\in \mathbb{R},$$
in the last two terms, we obtain
\begin{align}
\delta \mathcal{J}(\mathcal{F}) &= \int_\Omega (y(T, x, \mathcal{F})-y_{T}^\delta(x)) \delta y(T, x, \mathcal{F}) \,\d x +\frac{1}{2} \int_\Omega [\delta y(T, x, \mathcal{F})]^2 \,\d x \label{2eq2.19}\\
& \hspace{-0.3cm} + \int_\Gamma (y_\Gamma(T, x, \mathcal{F})-y_{T, \Gamma}^\delta(x)) \delta y_\Gamma(T, x, \mathcal{F}) \,\d S +\frac{1}{2} \int_\Gamma [\delta y_\Gamma(T, x, \mathcal{F})]^2 \,\d S \label{2eq2.20},
\end{align}
where
\begin{align*}
\delta y(T, \cdot, \mathcal{F})&= y(T, \cdot, \mathcal{F}+ \delta \mathcal{F})- y(T, \cdot, \mathcal{F}),\\
\delta y_\Gamma(T, \cdot, \mathcal{F}) &= y_\Gamma(T, \cdot, \mathcal{F}+ \delta \mathcal{F})-y_\Gamma(T, \cdot, \mathcal{F}).
\end{align*}
By linearity of the systems, $\delta Y=(\delta y, \delta y_\Gamma)$ is the mild solution of the following system
\begin{empheq}[left = \empheqlbrace]{alignat=2}
\begin{aligned}
&\partial_t (\delta y) -d \Delta (\delta y) + a(x)(\delta y) = \delta F(t,x), &\qquad \text{in } \Omega_T , \\
&\partial_t (\delta y_{\Gamma}) -\gamma \Delta_{\Gamma} (\delta y_{\Gamma})+ d \partial_{\nu} (\delta y) + b(x)(\delta y_{\Gamma}) = \delta G(t,x), &\qquad \text{on } \Gamma_T, \\
&\delta y_{\Gamma}(t,x) = (\delta y)_{|\Gamma}(t,x), &\qquad \text{on } \Gamma_T, \\
&(\delta y, \delta y_{\Gamma})\rvert_{t=0}=(0,0),   &\qquad \Omega \times\Gamma.
\label{deq1to4}
\end{aligned}
\end{empheq}
We rewrite the first integral in the right-hand side of \eqref{2eq2.19} using $\Phi(t,\cdot,\mathcal{F})$ and $\delta Y(t,\cdot, \mathcal{F})$, the mild solutions of \eqref{aeq1to5} and \eqref{deq1to4} respectively. We have
\begin{align}
&\int_\Omega (y(T, x, \mathcal{F})-y_{T}^\delta(x)) \delta y(T, x, \mathcal{F}) \,\d x \nonumber =\int_\Omega \varphi(T, x, \mathcal{F}) \delta y(T, x, \mathcal{F}) \,\d x \nonumber\\
&= \bigintss_\Omega \left[\int_0^T \partial_t (\varphi(t, x, \mathcal{F}) \delta y(t, x, \mathcal{F})) \,\d t \right]\,\d x \nonumber\\
&= \int_{\Omega_T} \left[(\partial_t \varphi) \delta y + \varphi \partial_t(\delta y) \right] \,\d x \,\d t \nonumber\\
&= \int_{\Omega_T} \left[ (-d \Delta \varphi +a(x)\varphi) \delta y + \varphi (d \Delta (\delta y)-a(x)(\delta y)) \right]\,\d x \,\d t \nonumber\\
& \qquad + \int_{\Omega_T} \delta F(t,x)\,\varphi(t,x) \,\d x \,\d t \nonumber\\
&= \int_{\Omega_T} -d[(\Delta \varphi) \delta y -\Delta(\delta y)\varphi ]\,\d x \, \d t + \int_{\Omega_T} \delta F(t,x)\,\varphi(t,x) \,\d x \, \d t \nonumber\\
&= \int_{\Gamma_T} -d[(\partial_\nu \varphi) \delta y_\Gamma -\varphi_\Gamma \partial_\nu  (\delta y)] \,\d S \, \d t + \int_{\Omega_T} \delta F(t,x)\,\varphi(t,x) \,\d x \, \d t \label{2eq2.25}.
\end{align}
Similarly, for the first integral in the right-hand side of \eqref{2eq2.20}, we obtain
\begin{align}
&\int_\Gamma (y_\Gamma(T, x, \mathcal{F})-y_{T, \Gamma}^\delta(x)) \delta y_\Gamma(T, x, \mathcal{F}) \,\d S \nonumber\\
&= \int_{\Gamma_T} -\gamma[(\Delta_\Gamma \varphi_\Gamma) \delta y_\Gamma -\Delta_\Gamma(\delta y_\Gamma)\varphi_\Gamma] \,\d S \,\d t \label{2eq2.26}\\
&\qquad + \int_{\Gamma_T} d[(\partial_\nu \varphi) \delta y_\Gamma -\varphi_\Gamma \partial_\nu  (\delta y)] \,\d S \,\d t + \int_{\Gamma_T} \delta G(t,x)\,\varphi_\Gamma(t,x) \,\d S \,\d t \nonumber\\
&= \int_{\Gamma_T} d[(\partial_\nu \varphi) \delta y_\Gamma -\varphi_\Gamma \partial_\nu  (\delta y)] \,\d S \,\d t+ \int_{\Gamma_T} \delta G(t,x)\,\varphi_\Gamma(t,x) \,\d S \,\d t \label{2eq2.27}
\end{align}
The first integral in the right-hand side of \eqref{2eq2.26} is null by the surface divergence formula \eqref{sdt}. Adding up the two integrals \eqref{2eq2.25} and \eqref{2eq2.27}, we obtain
\begin{align}
&\int_\Omega (y(T, x, \mathcal{F})-y_{T}^\delta(x)) \delta y(T, x, \mathcal{F}) \,\d x + \int_\Gamma (y_\Gamma(T, x, \mathcal{F})-y_{T, \Gamma}^\delta(x)) \delta y_\Gamma(T, x, \mathcal{F}) \,\d S \nonumber\\
&= \int_{\Omega_T} \delta F(t,x)\,\varphi(t,x) \,\d x \, \d t + \int_{\Gamma_T} \delta G(t,x)\,\varphi_\Gamma(t,x) \,\d S \,\d t. \nonumber
\end{align}
For the second integrals in the right-hand side of \eqref{2eq2.19} and \eqref{2eq2.20}, the estimate \eqref{02eq2.40} implies
$$\frac{1}{2} \int_\Omega [\delta y(T, x, \mathcal{F})]^2 \,\d x +\frac{1}{2} \int_\Gamma [\delta y_\Gamma(T, x, \mathcal{F})]^2 \,\d S =\mathcal{O}\left(\|\delta \mathcal{F}\|^2_{\mathbb{L}^2}\right).$$
This completes the proof.
\end{proof}

Next, we prove the Lipschitz continuity of the Fréchet gradient $\mathcal{J}'$, in particular, $\mathcal{J}\in C^1(\mathcal{U})$.
\begin{lemma}\label{2lem2.8}
Let $\mathcal{F}, \delta \mathcal{F}\in \mathcal{U}$. Then the Fréchet gradient $\mathcal{J}'$ is Lipschitz continuous,
\begin{equation*}
\|\mathcal{J}'(\mathcal{F}+ \delta \mathcal{F})-\mathcal{J}'(\mathcal{F})\|_{\mathbb{L}^2} \leq L \|\delta \mathcal{F}\|_{\mathbb{L}^2}, 
\end{equation*}
where the Lipschitz constant $L>0$ depends on $T,\|a\|_\infty$ and $\|b\|_\infty$ as follows
\begin{equation}\label{lip}
L=\sqrt{2 T \mathrm{e}^{(1+4\max(\|a\|_\infty, \|b\|_\infty))T}}. 
\end{equation}
\end{lemma}

\begin{proof}
Let $\delta \Phi(t, \cdot, \mathcal{F}):=(\delta \varphi, \delta \varphi_\Gamma)$ be the strong solution of the adjoint system
\begin{empheq}[left = \empheqlbrace]{alignat=2}
\begin{aligned}
&-\partial_t (\delta \varphi) -d \Delta (\delta \varphi)+ a(x)(\delta \varphi) = 0, &\qquad \text{in } \Omega_T , \\
&-\partial_t (\delta \varphi_{\Gamma}) -\gamma \Delta_{\Gamma} (\delta \varphi_{\Gamma}) + d\partial_{\nu} (\delta \varphi) + b(x)(\delta \varphi_{\Gamma}) = 0, &\qquad \text{on } \Gamma_T, \\
&\delta \varphi_{\Gamma}(t,x) = (\delta \varphi)_{|\Gamma}(t,x), &\qquad \text{on } \Gamma_T, \\
&(\delta \varphi, \delta \varphi_{\Gamma})\rvert_{t=T}=(\delta y, \delta y_{\Gamma})\rvert_{t=T},   &\qquad \Omega \times\Gamma.
\label{2aeq1to4}
\end{aligned}
\end{empheq}
Using Proposition \ref{2prop2.6}, we have
\begin{align}
\|\mathcal{J}'(\mathcal{F}+ \delta \mathcal{F})-\mathcal{J}'(\mathcal{F})\|^2_{\mathbb{L}^2_T} &=\int_{\Omega_T} |\delta \varphi|^2 \,\d x \, \d t + \int_{\Gamma_T} |\delta \varphi_\Gamma|^2 \,\d S\,\d t \nonumber\\
& \leq 2 \left(\|\delta \varphi\|^2_{L^2(\Omega_T)} +  \|\delta \varphi_\Gamma\|^2_{L^2(\Gamma_T)}\right) \nonumber\\
& = 2 \|\delta \Phi\|_{\mathbb{L}^2_T}^2. \label{2eq2.35}
\end{align}
Next, we estimate the norm $\|\delta \Phi\|_{\mathbb{L}^2_T}^2$. In the adjoint system \eqref{2aeq1to4}, multiplying the first equation by $\delta \varphi$ and the second by $\delta \varphi_\Gamma$, we obtain the following identities
\begin{align*}
& -\frac{1}{2} \partial_t \left( \int_\Omega |\delta \varphi|^2 \,\d x\right) +d\int_\Omega |\nabla(\delta \varphi)|^2 \,\d x -\frac{1}{2} \int_\Gamma d\partial_\nu (|\delta \varphi|^2) \,\d S + \int_\Omega a |\delta \varphi|^2 \,\d x =0, \\
& -\frac{1}{2} \partial_t \left( \int_\Gamma |\delta \varphi_\Gamma|^2 \,\d S\right) +\gamma  \int_\Gamma |\nabla_\Gamma(\delta \varphi_\Gamma)|^2 \,\d S +\frac{1}{2} \int_\Gamma d\partial_\nu (|\delta \varphi|^2) \,\d S + \int_\Gamma b |\delta \varphi_\Gamma|^2 \,\d S =0, 
\end{align*}
where we used integration by parts and the surface divergence formula \eqref{sdt}. Adding up the last two equalities, we obtain
\begin{align*}
&\frac{1}{2} \partial_t \left( \int_\Omega |\delta \varphi|^2 \,\d x + \int_\Gamma |\delta \varphi_\Gamma|^2 \,\d S\right) -\int_\Omega a |\delta \varphi|^2 \,\d x -\int_\Gamma b |\delta \varphi_\Gamma|^2 \,\d S \nonumber\\
&= d\int_\Omega |\nabla(\delta \varphi)|^2 \,\d x + \gamma  \int_\Gamma |\nabla_\Gamma(\delta \varphi_\Gamma)|^2 \,\d S .
\end{align*}
Then
\begin{align*}
&\frac{1}{2} \partial_t \left( \int_\Omega |\delta \varphi|^2 \,\d x + \int_\Gamma |\delta \varphi_\Gamma|^2 \,\d S\right) + \max(\|a\|_\infty,\|b\|_\infty)\left( \int_\Omega |\delta \varphi|^2 \,\d x + \int_\Gamma |\delta \varphi_\Gamma|^2 \,\d S\right)\geq 0.
\end{align*}
This inequality implies that the function $H$ defined by
\begin{equation*}
H(t)=\mathrm{e}^{2\max(\|a\|_\infty,\|b\|_\infty) t} \left( \int_\Omega |\delta \varphi|^2 \,\d x + \int_\Gamma |\delta \varphi_\Gamma|^2 \,\d S\right)
\end{equation*}
is nondecreasing on $[0,T]$. Using $\eqref{2aeq1to4}_4$, it follows, for all $t\in [0,T]$, that
\begin{align*}
&\|\delta \Phi\|^2_{\mathbb{L}^2_T} =\int_{\Omega_T} |\delta \varphi(t,x,\mathcal{F})|^2 \,\d x \, \d t + \int_{\Gamma_T} |\delta \varphi_\Gamma(t,x,\mathcal{F})|^2 \,\d S \,\d t \nonumber\\
& \leq T \mathrm{e}^{2 \max(\|a\|_\infty,\|b\|_\infty) T} \left( \int_\Omega |\delta \varphi(T,x,\mathcal{F})|^2 \,\d x + \int_\Gamma |\delta \varphi_\Gamma(T,x,\mathcal{F})|^2 \,\d S\right)\nonumber\\
& \le T \mathrm{e}^{2 \max(\|a\|_\infty,\|b\|_\infty) T} \left( \int_\Omega |\delta y(T,x,\mathcal{F})|^2 \,\d x + \int_\Gamma |\delta y_\Gamma(T,x,\mathcal{F})|^2 \,\d S \right).
\end{align*}
Using this last inequality and inequality \eqref{02eq2.40}, we obtain
\begin{align*}
2\|\delta \Phi\|^2_{\mathbb{L}^2_T} &\leq 2 T \mathrm{e}^{2 \max(\|a\|_\infty,\|b\|_\infty) T} \mathrm{e}^{(1+2\max(\|a\|_\infty, \|b\|_\infty))T} \|\delta \mathcal{F}\|^2_{\mathbb{L}^2_T}\nonumber\\
\qquad & \leq 2 T \mathrm{e}^{(1+4\max(\|a\|_\infty, \|b\|_\infty))T} \|\delta \mathcal{F}\|^2_{\mathbb{L}^2_T}.
\end{align*}
With the estimate \eqref{2eq2.35} this implies that
\begin{equation*}
\|\mathcal{J}'(\mathcal{F}+ \delta \mathcal{F})-\mathcal{J}'(\mathcal{F})\|^2_{\mathbb{L}^2_T} \leq 2 T \mathrm{e}^{(1+4\max(\|a\|_\infty, \|b\|_\infty))T}\|\delta \mathcal{F}\|^2_{\mathbb{L}^2_T}.
\end{equation*}
This yields the desired result. 
\end{proof}

Next, we consider the Landweber method given by the following iteration
\begin{equation}\label{2eq2.43}
\mathcal{F}_{k+1}= \mathcal{F}_k-\alpha_k \mathcal{J}'(\mathcal{F}_k), \quad k=0,1,2,\dots,
\end{equation}
where $\mathcal{F}_0\in \mathcal{U}$ is a given initial iteration and $\alpha_k$ is a relaxation parameter defined by the minimization problem
$$h_k(\alpha):=\inf_{\alpha\ge 0} h_k(\alpha), \quad h_k(\alpha):=\mathcal{J}\left(\mathcal{F}_k-\alpha \mathcal{J}'(\mathcal{F}_k)\right), \quad k=0,1,2,\dots$$
We refer to \cite{EHN'00} for a detailed exposition of this method.

\begin{remark}
At this level, some remarks should be made:
\begin{itemize}
\item[$\bullet$] The Lipschitz continuity of the gradient $\mathcal{J}'$ implies that the sequence $\left(\mathcal{J}(\mathcal{F}_k)\right)$ is decreasing, where $\left(\mathcal{F}_k\right)$ is defined by \eqref{2eq2.43}. This fact yields fast numerical results; see \cite[Lemma 3.4.4]{HR'17} for more details.
\item[$\bullet$] In some situations, the choice of the iteration parameter $\alpha_k >0$ is difficult. However, the Lipschitz continuity of $\mathcal{J}'$ allows us to estimate this parameter by the Lipschitz constant $L$ via the estimate
\begin{equation*}
0<\lambda_0 \leq \alpha_k \leq \frac{2}{L+2\lambda_1}
\end{equation*}
for arbitrary parameters $\lambda_0,\lambda_1 >0$; see \cite[Section 3.4.3]{Ha'07}.
\item[$\bullet$] If $\alpha_k=\alpha=\mathrm{const}>0$ for all $k$, the optimal value of $\alpha$ (which corresponds to $\lambda_0=\frac{1}{L}$ and $\lambda_1=\frac{L}{2}$) is $\alpha_*=\frac{1}{L}$. Hence, if $L$ is large, the step parameter $\alpha_*$ will be small. This fact illustrates the importance of a sharp Lipschitz constant $L$.
\end{itemize}
\end{remark}

The next lemma follows the same ideas in Corollary 4.1 and Theorem 4.1 in \cite{Ha'07} (one can also refer to \cite{Va'81}). We denote by $\mathcal{U}_*$ the set of all quasi-solutions of \eqref{neq2.5}-\eqref{neq2.6}.
\begin{lemma}
Let $(\mathcal{F}_k) \subset \mathcal{U}$ be the sequence defined by \eqref{2eq2.43}. If the iteration parameter $\alpha_k =\alpha_{*}$ for all $k$. Then the following assertions hold.
\begin{enumerate}
\item[(i)] The sequence $(\mathcal{J}(\mathcal{F}_k))$ is monotone decreasing and convergent. Moreover,
\begin{equation*}
\lim_{k \to \infty} \|\mathcal{J}'(\mathcal{F}_k)\|_{\mathbb{L}^2_T} =0. 
\end{equation*}
Moreover,
\begin{equation*}
\|\mathcal{F}_{k+1} -\mathcal{F}_k\|_{\mathbb{L}^2_T}^2 \le \frac{2}{L} [\mathcal{J}(\mathcal{F}_k)-\mathcal{J}(\mathcal{F}_{k+1})], \quad k=1,2,3, \ldots,
\end{equation*}
where $L$ is the Lipschitz constant in \eqref{lip};
\item[(ii)] for any given initial iteration $\mathcal{F}_0 \in \mathcal{U}$,  the sequence $(\mathcal{F}_k)$ converges weakly in $\mathbb{L}^2$ to a quasi-solution $\mathcal{F}_* \in \mathcal{U}_*$ of (ISP). The rate of convergence of $(\mathcal{J}(\mathcal{F}_k))$ can be estimated as follows
$$
0 \leq \mathcal{J}(\mathcal{F}_k)-\mathcal{J}_* \leq 2 L \frac{\beta^2}{k}, \quad k=1,2,3, \ldots,
$$
where $\mathcal{J}_* :=\lim\limits_{k \to \infty} \mathcal{J}\left(\mathcal{F}_k\right)$ and $\beta:=\sup \{\left\|\mathcal{F}_k-\mathcal{F}_*\right\|_{\mathbb{L}^2_T} \colon \mathcal{F}_k\in \mathcal{U}, \mathcal{F}_* \in \mathcal{U}_*\}.$
\end{enumerate}
\end{lemma}

\section{Existence and Uniqueness of the solution to (ISP)}\label{sec4}
In the following, we use some tools from the calculus of variations to study the existence and uniqueness of the solution to (ISP). First, we prove the following lemma.
\begin{lemma}\label{monlem}
For the cost functional $\mathcal{J}\in C^1(\mathcal{U})$, the following formula holds
\begin{equation}\label{Moneq}
\langle \mathcal{J}'(\mathcal{F} + \delta \mathcal{F})- \mathcal{J}'(\mathcal{F}),\delta \mathcal{F}\rangle_{\mathbb{L}^2_T}=2 \|\delta Y(T,\cdot, \mathcal{F})\|^2_{\mathbb{L}^2}, \quad \forall \mathcal{F}, \delta \mathcal{F} \in \mathcal{U},
\end{equation}
where $\delta Y(T,\cdot, \mathcal{F})$ is the solution of \eqref{deq1to4}.
\end{lemma}

\begin{proof}
Let $(\delta \varphi, \delta \varphi_\Gamma)$ be the solution of \eqref{2aeq1to4}. By the gradient formula \eqref{2eq2.11}, we have
\begin{align*}
&\langle \mathcal{J}'(\mathcal{F} + \delta \mathcal{F})- \mathcal{J}'(\mathcal{F}),\delta \mathcal{F}\rangle_{\mathbb{L}^2_T} = \int_{\Omega_T} \delta F \delta \varphi \,\d x\, \d t + \int_{\Gamma_T} \delta G \delta \varphi_\Gamma \,\d S\, \d t \notag\\
& = \int_{\Omega_T} \partial_t(\delta y) \delta \varphi \,\d x\, \d t + \int_{\Gamma_T} \partial_t(\delta y_\Gamma) \delta \varphi_\Gamma \,\d S\, \d t + \int_{\Omega_T} [-d \Delta(\delta y) + a(x)(\delta y)]\delta \varphi \,\d x\, \d t\notag\\
& \qquad + \int_{\Gamma_T} [-\gamma \Delta(\delta y_\Gamma) + d \partial_\nu (\delta y) + b(x)(\delta y_\Gamma)]\delta \varphi_\Gamma \,\d S\, \d t \notag\\
& = \|\delta y(T,\cdot)\|^2_{L^2(\Omega)} + \|\delta y_\Gamma(T,\cdot)\|^2_{L^2(\Gamma)} - \int_{\Omega_T} \partial_t(\delta \varphi) \delta y \,\d x\, \d t - \int_{\Gamma_T} \partial_t(\delta \varphi_\Gamma) \delta y_\Gamma \,\d S\, \d t \notag\\
& + \int_{\Omega_T} [-d \Delta(\delta \varphi) + a(x)(\delta \varphi)]\delta y \,\d x\, \d t + \int_{\Gamma_T} [-\gamma \Delta(\delta \varphi_\Gamma) + d \partial_\nu (\delta \varphi) + b(x)(\delta \varphi_\Gamma)]\delta y_\Gamma \,\d S\, \d t \notag\\
& = \|\delta y(T,\cdot)\|^2_{L^2(\Omega)} + \|\delta y_\Gamma(T,\cdot)\|^2_{L^2(\Gamma)},
\end{align*}
where we used integration by parts with respect to $t$, the Green formula in $\Omega$, and the surface divergence formula on $\Gamma$, together with the system \eqref{deq1to4}.
\end{proof}

The monotonicity of $\mathcal{J}' \colon \mathcal{U} \rightarrow \mathbb{L}^2_T$ in Lemma \ref{monlem} implies the convexity of the functional $\mathcal{J}$. As a direct consequence of Lemma \ref{lemCont} and \cite[Theorem 25.C]{Ze'90}, we have the following existence result.
\begin{corollary}
The cost functional $\mathcal{J}$ is continuous and convex on $\mathcal{U}$. There exists then a minimizer $\hat{\mathcal{F}} \in \mathcal{U}$ such that
$$\mathcal{J}\left(\hat{\mathcal{F}}\right)=\min_{\mathcal{F}\in \mathcal{U}} \mathcal{J}(\mathcal{F}).$$
\end{corollary}

Since the strict convexity of $\mathcal{J}$ is characterized by the strict monotonicity of $\mathcal{J}'$, the equality \eqref{Moneq} yields a sufficient condition for uniqueness.
\begin{lemma}
If the positivity condition
\begin{equation*}
\int_{\Omega} (\delta y(T, x; \mathcal{F}))^2 \,\d x + \int_{\Gamma} (\delta y_\Gamma(T, x; \mathcal{F}))^2 \,\d S >0, \qquad \forall \mathcal{F} \in \mathcal{U},
\end{equation*}
holds, then the problem (ISP) admits at most one solution.
\end{lemma}
The above lemma is a consequence of the uniqueness theorem for strictly convex functionals defined on convex sets (see, e.g.,\cite[Corollary 25.15]{Ze'90}).

\section{Numerical simulation for 1-D internal heat source}\label{sec5}
In this section, we analyze the numerical reconstruction for an unknown source term that depends only on the space variable. More precisely, we consider the reconstruction of $f(x)$ in the following 1-D heat equation with dynamic boundary conditions
\begin{empheq}[left = \empheqlbrace]{alignat=2}
\begin{aligned}
&y_{t}(t, x)-y_{x x}(t, x)=f(x)r(t,x), &&\qquad(t,x)\in (0,T)\times (0,\ell)  , \\
&y_t(t, 0) - y_{x}(t, 0)=0, &&\qquad t\in (0,T), \\
&y_t(t, \ell) + y_{x}(t, \ell)=0, &&\qquad t\in (0,T), \\
&y(0,x)=y_0(x), \quad (y(0,0), y(0, \ell))=(a,b), &&\qquad x\in (0,\ell),
\label{1deq1to4}
\end{aligned}
\end{empheq}
where $T>0$ is a final time, $\ell>0$ and $\left(y_0,a,b\right)\in L^2(0,\ell) \times\mathbb{R}^2$ is an initial condition. The function $r \in C^1\left([0,T]; C\left([0,\ell]\right)\right)$ is assumed to be known.

The 1-D heat equation with dynamic boundary conditions has attracted special attention as a model of heat conduction problems for a metal bar of length $\ell$. We mention \cite{KN'04}, where the authors have studied the approximate controllability problem from the boundary.

We shall apply the Landweber iteration scheme discussed in a previous section to system \eqref{1deq1to4}. Let $Y(t,x,f):=\left(y(t,x),y(t,0),y(t,\ell)\right)$ be the solution of \eqref{1deq1to4}. The input-output operator $\Psi \colon L^2(0,\ell) \longrightarrow L^2(0,\ell)\times \mathbb{R}^2$ is given by
$$(\Psi f)(x):=Y(T,x,f)=\left(y(T,x),y(T,0),y(T,\ell)\right), \qquad x\in (0,\ell),$$
and the Tikhonov functional by
\begin{align*}
J_\varepsilon(f)&=\frac{1}{2} \left\|Y(T, \cdot,f)-Y_{T}^\delta \right\|_{L^2(0,\ell)\times \mathbb{R}^2}^2 +\frac{\varepsilon}{2} \|f\|_{L^2(0,\ell)}^2, \qquad f\in L^2(0,\ell), \\
& \hspace{-0.3cm}= \frac{1}{2} \left(\left\|y(T, \cdot)-y_{T}^\delta\right\|_{L^2(0,\ell)}^2 + \left|y(T,0)-y_T^{0,\delta}\right|^2 + \left|y(T,\ell)-y_T^{\ell,\delta}\right|^2 + \varepsilon \|f\|_{L^2(0,\ell)}^2\right), \notag
\end{align*}
where $Y_{T}^\delta:=\left(y_{T}^\delta, y_T^{0,\delta}, y_T^{\ell,\delta}\right) \in L^2(0,\ell)\times \mathbb{R}^2$.
The adjoint system corresponding to \eqref{aeq1to5} is given by
\begin{empheq}[left = \empheqlbrace]{alignat=2}
\begin{aligned}
& \varphi_{t}(t, x)+\varphi_{x x}(t, x)=0, &&\hspace{-1cm} (t,x)\in (0,T)\times (0,\ell)  , \\
& \varphi_t(t, 0) + \varphi_{x}(t, 0)=0, && t\in (0,T), \\
& \varphi_t(t, \ell) - \varphi_{x}(t, \ell)=0, && t\in (0,T), \\
& \varphi(T, x)=y(T,x)-y_T^\delta(x), && x\in (0,\ell),\\
& (\varphi(T,0), \varphi(T, \ell))=\left(y(T,0)-y_T^{0,\delta},y(T,\ell)-y_T^{\ell,\delta}\right).
\label{1daeq1to4}
\end{aligned}
\end{empheq}
Similarly to calculations in Section \ref{sec3}, the gradient of $J_\varepsilon$ is given by
\begin{equation}\label{1dj'}
J_\varepsilon'(f)(x)=\int_0^T \varphi(t,x,f) r(t,x) \,\d t + \varepsilon f(x), \quad f\in L^2(0,\ell),\;  x\in (0,\ell).
\end{equation}

The gradient formula \eqref{1dj'} for the Tikhonov functional $J_\varepsilon$ allows us to implement classical versions of the Conjugate Gradient Algorithm with different conjugation coefficients. Here, we shall restrict ourselves to the following basic algorithm.
\smallskip

\begin{algorithm}[H]\label{alg1}
\SetAlgoLined
 Set $k=0$ and choose an initial source $f_0$\;
 Solve the direct problem \eqref{1deq1to4} to obtain $Y(t,x,f_k)$\;
 Knowing the computed $Y(T,x,f_k)$ and the measured $Y_T^\delta$, solve the adjoint problem \eqref{1daeq1to4} to obtain $\varphi(t,x,f_k)$\;
 Compute the descent direction $p_k=J_\varepsilon'(f_k)$ using \eqref{1dj'}\;
 Solve the direct problem \eqref{1deq1to4} with source $p_k$ to get the solution $\Psi p_k$\;
 Compute the relaxation parameter $\displaystyle \alpha_k =\frac{\|p_k\|_{L^2(0,\ell)}^2}{\|\Psi p_k\|_{L^2(0,\ell)\times \mathbb{R}^2}^2}$ (see \cite[Lemma 3.4.1]{HR'17})\;
 Find the next iteration $f_{k+1}=f_k- \alpha_k p_k$\;
 Stop the iteration process if the stopping criterion $J_\varepsilon(f_{k+1}) <e_J$ holds. Otherwise, set $k:=k+1$ and go to Step 2\;
 \caption{Landweber iteration scheme}
\end{algorithm}
\smallskip

Since we deal with dynamic boundary conditions that contain time derivative, the numerical solutions for the direct problem \eqref{1deq1to4} and the adjoint problem \eqref{1daeq1to4} will be obtained by using the method of lines. This method is implemented in our case with help of the \texttt{Mathematica} system \cite{Wo'05} via the function NDSolve for solving ordinary differential systems. The noisy terminal data will be generated from the exact output data as follows
$$Y_T^\delta(x) =Y_T(x) + p \times\|Y_T\|_{L^2(0,\ell)\times \mathbb{R}^2} \times \mathrm{RandomReal[]},$$
where $p$ stands for the percentage of the noise level, and the function RandomReal[] produces random real numbers.

In the subsequent numerical tests, we will choose, for simplicity, the following values for the known parameters
$$T=1, \quad \ell=1, \quad r=1, \quad y_0=0, \quad a=b=0.$$

\noindent\textbf{Example 1}\\
We consider the following basic source term $f(x)=x(1-x), \; x\in (0,1)$. First, we compute the numerical solution $Y(t,\cdot,f)$ to generate the data $Y(1,\cdot,f)$. Next, we plot the corresponding solution.

\begin{figure}[H]
\centering
\includegraphics[scale=0.5]{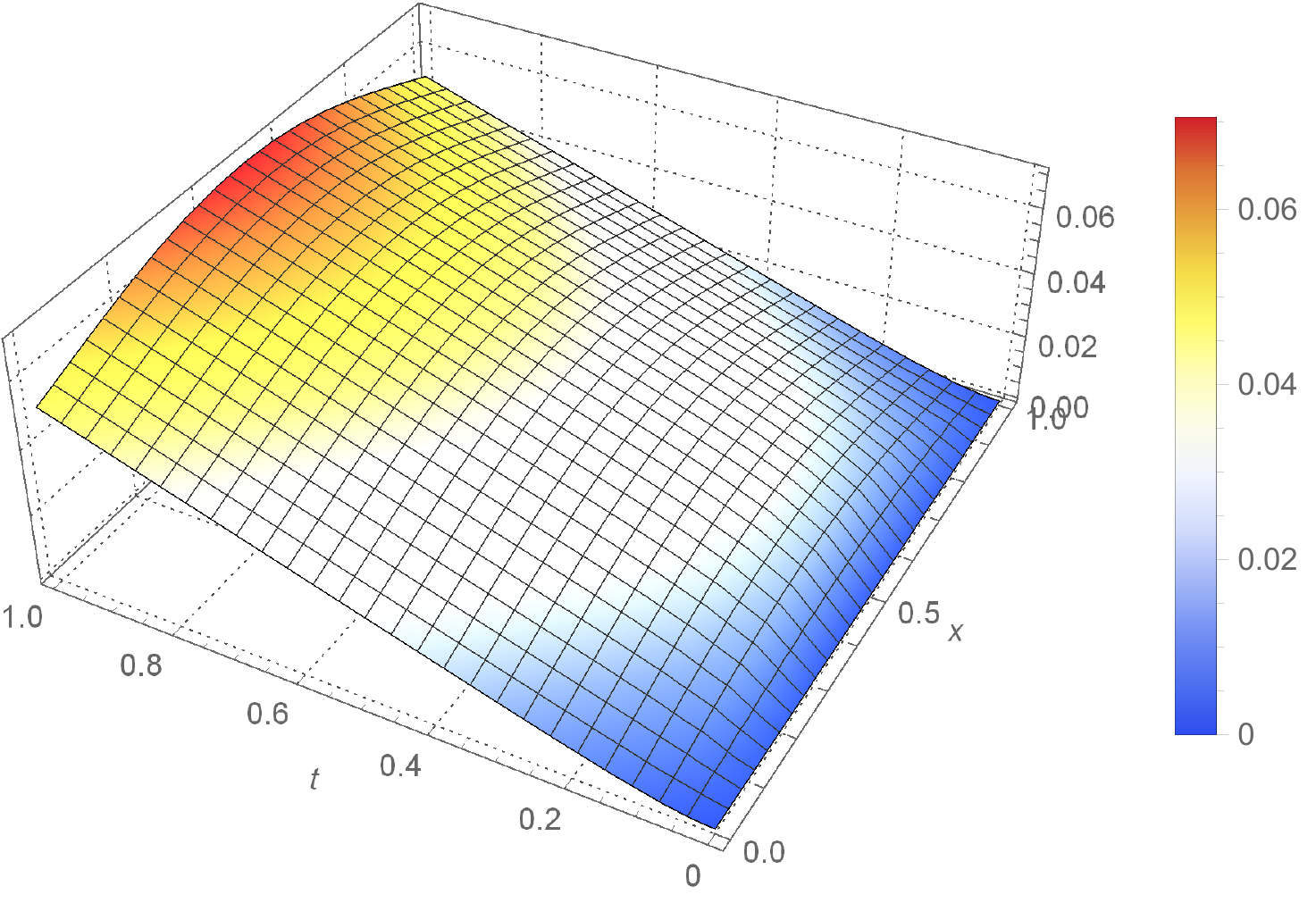}
\caption{Solution $Y(t,x,f)$ of \eqref{1deq1to4}.}
\end{figure}

\begin{figure}[H]
\centering
\includegraphics[scale=0.5]{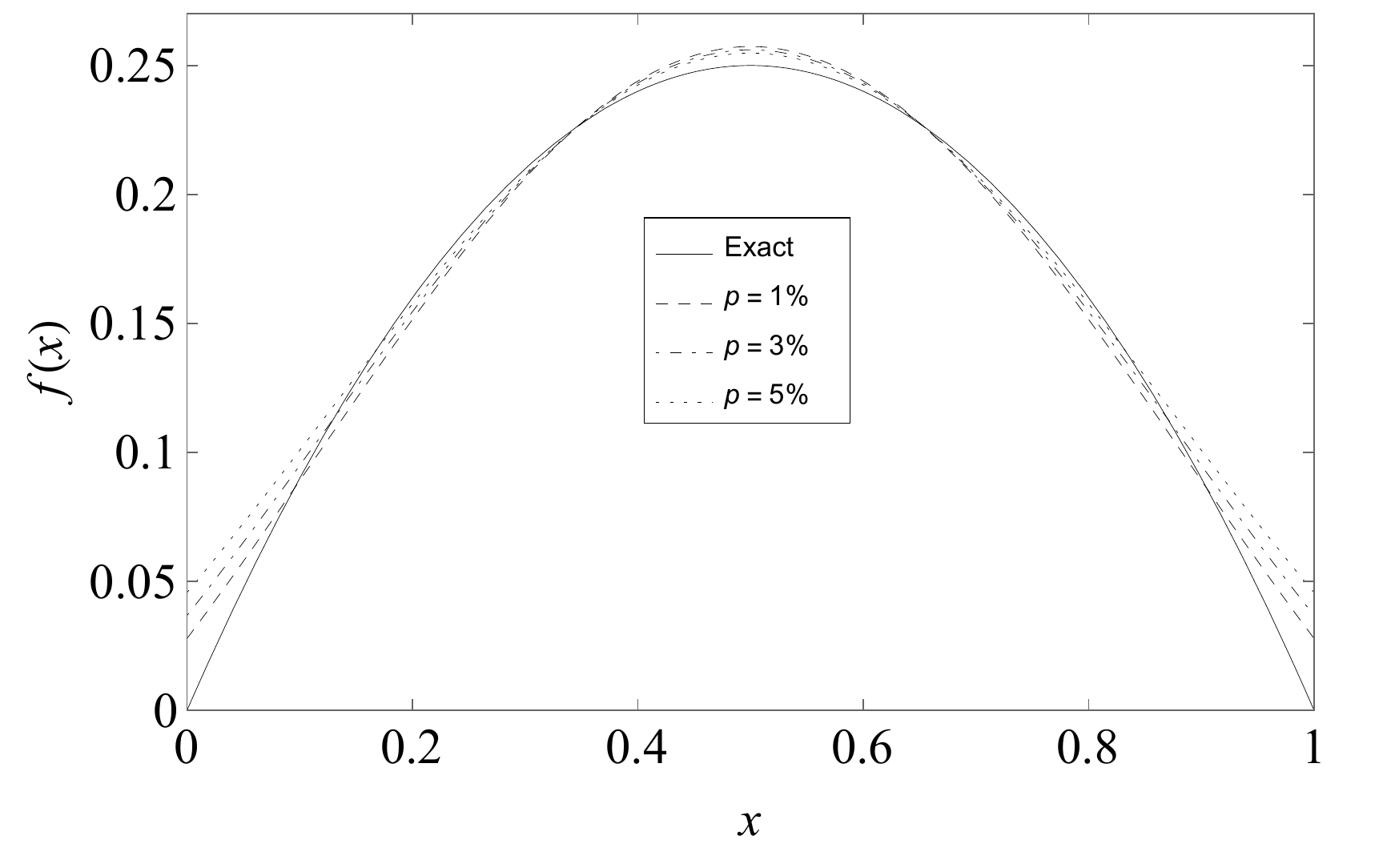}
\caption{Exact and recovered $f(x)$ by using Algorithm \ref{alg1}, for $p\in \{1\%,3\%,5\%\}$, respectively.}
\end{figure}

The initial iteration is chosen as $f_0=0$ with regularization parameter and stopping parameter $\varepsilon=e_J=10^{-6}$. The algorithm stops at iterations $k\in \{2,2,2\}$, for $p\in \{1\%,3\%,5\%\}$, respectively.

\begin{remark}
The optimal regularization parameter $\varepsilon^{\mathrm{opt}}$ can be defined from the conditions $\varepsilon <1$ and $\frac{\delta^2}{\varepsilon}<1$ depending on the noise level $\delta$. The optimal stopping parameter $e_J^{\mathrm{opt}}$ can be defined by analyzing the behavior of the convergence error and the accuracy error defined respectively by
\begin{equation}\label{err}
\begin{aligned}
e(k,f_k)&:=\|\Psi f_k -Y_T\|_{L^2(0,1)\times \mathbb{R}^2}^2\\
E(k,f_k)&:=\|f-f_k\|_{L^2(0,1)},
\end{aligned}
\end{equation}
in terms of the iteration number $k$. We refer to \cite{HR'17} for more details.
\end{remark}

\noindent\textbf{Example 2}\\
We take the exact source term as $f(x)=\sin(\pi x), \; x\in (0,1)$.

\begin{figure}[H]
\centering
\includegraphics[scale=0.5]{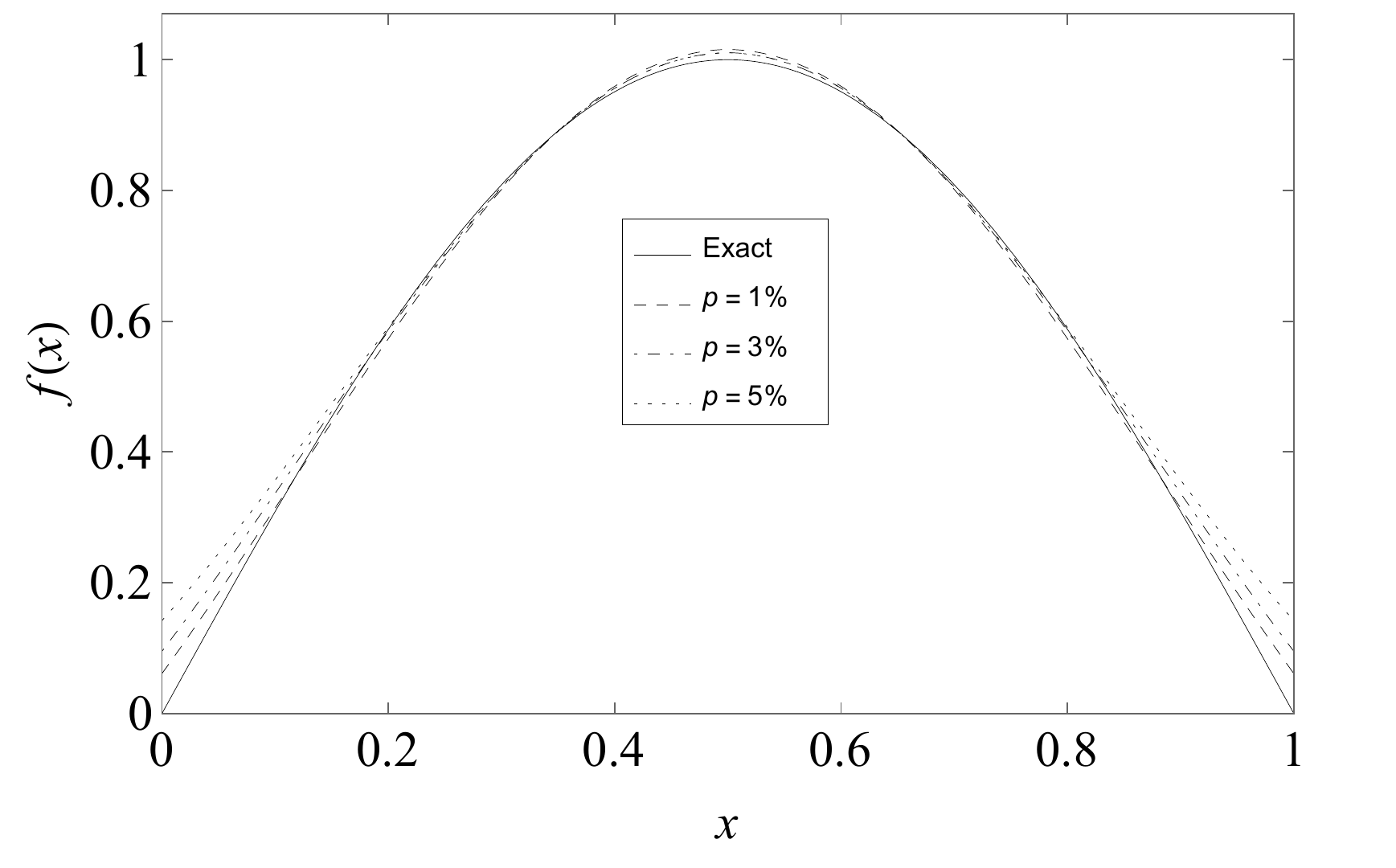}
\caption{Exact and recovered $f(x)$ by using Algorithm \ref{alg1}, for $p\in \{1\%,3\%,5\%\}$, respectively.}
\end{figure}
The initial iteration is chosen as $f_0=0$ with regularization parameter and stopping parameter $\varepsilon=e_J=10^{-8}$. The algorithm stops at iterations $k\in \{3,4,4\}$, for $p\in \{1\%,3\%,5\%\}$, respectively.
\vspace{0.5cm}

\noindent\textbf{Example 3}\\
We take the exact source term as $f(x)=\exp\left(-8\left(x-\frac{1}{2}\right)^2\right), \; x\in (0,1)$.

\begin{figure}[H]
\centering
\includegraphics[scale=0.5]{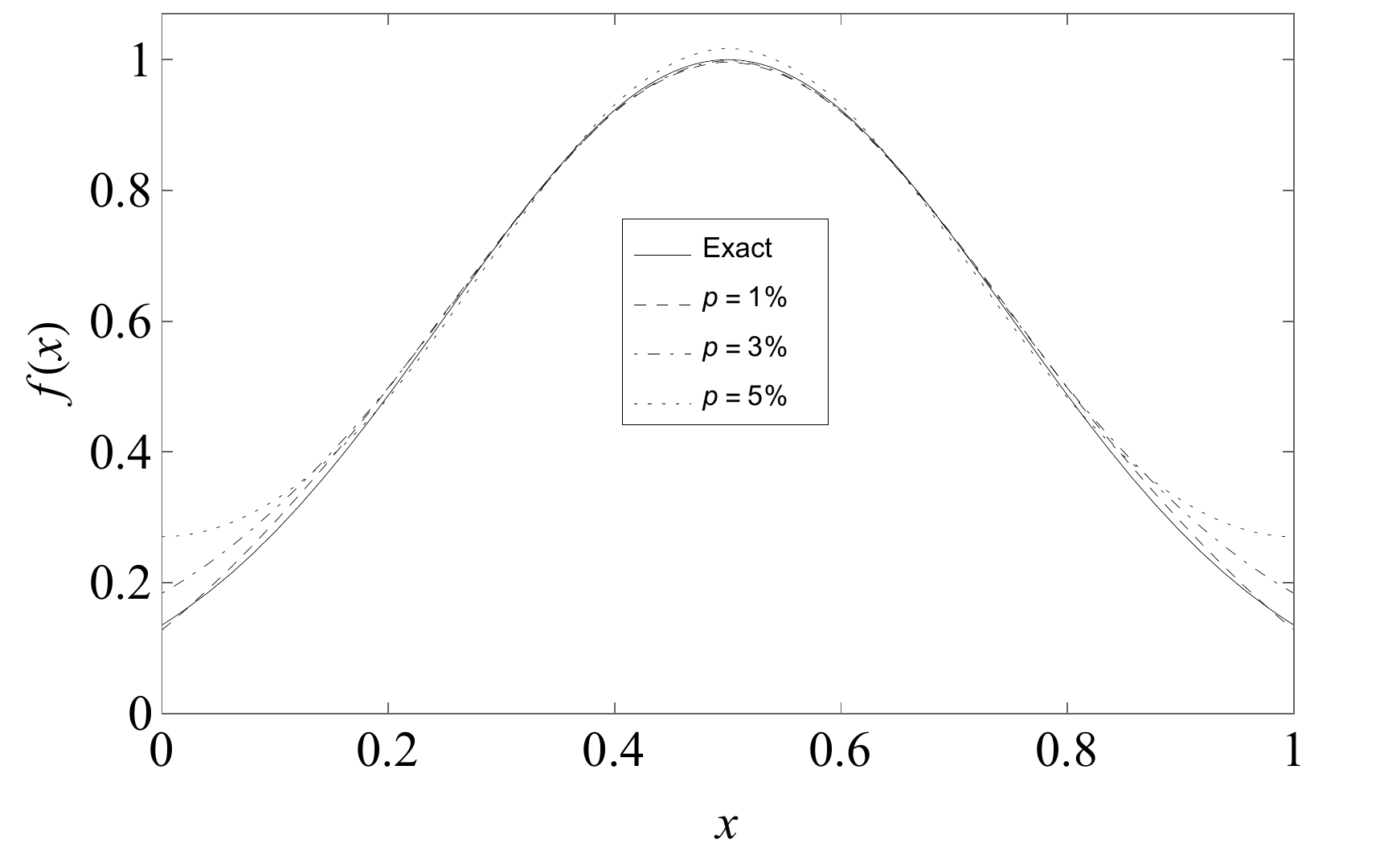}
\caption{Exact and recovered $f(x)$ by using Algorithm \ref{alg1} for $p\in \{1\%,3\%,5\%\}$, respectively.}
\end{figure}
The initial iteration is chosen as $f_0=0$ with regularization parameter and stopping parameter $\varepsilon=e_J=10^{-8}$. The algorithm stops at iterations $k\in \{5,5,35\}$ for $p\in \{1\%,3\%,5\%\}$, respectively.

\begin{table}[ht]
\caption{Errors depending on the iteration number $k$ for noise free data ($p=0\%$).}
\centering\small
$\begin{array}{cccccc}
\hline
 k & 1 & 2 & 3 & 4 & 5 \\
\specialrule{.1em}{.1em}{.1em}\\[-3mm]
e\left(k,f_k\right) & 3.421\times 10^{-1} & 1.101\times 10^{-2} & 4.517\times 10^{-4} & 2.767\times 10^{-4} & 1.413\times 10^{-4} \\
\hline\\[-3mm]
E\left(k,f_k\right) & 2.839\times 10^{-1} & 3.471\times 10^{-2} & 3.469\times 10^{-2} & 9.412\times 10^{-3} & 9.409\times 10^{-3} \\
\hline
\end{array}$
\end{table}

\begin{remark}
The above numerical experiments show that the Landweber scheme yields stable, accurate and fast results for the reconstruction of unknown source terms in heat equation with dynamic boundary conditions. We clearly see that the recovery of the source $f(x)$ becomes more accurate as the noise level $p$ decreases. This approach can be adapted for the simultaneous recovery of internal and boundary source terms that depend on both time and space, but it might require a large number of iterations. A performance analysis for this case will be treated in a forthcoming paper.
\end{remark}

\section{Conclusions}\label{sec6}
In this paper, we have considered an inverse problem for determining internal and boundary source terms from final time data in heat equation with dynamic boundary conditions. Adapting the weak solution approach, a minimization problem for the Tikhonov functional is analyzed, and a gradient formula of the functional is established via the solution of an appropriate adjoint system. Then the Lipschitz continuity of the Fréchet gradient is proved. Using calculus of variations techniques, the existence and the uniqueness of a quasi-solution are proved. In particular, a sufficient condition for the uniqueness is presented. Finally, some numerical tests are provided for recovering an internal heat source in the one-dimensional case.

\end{document}